\documentclass[12pt]{article}
\usepackage[a4paper]{geometry}
\usepackage{hyperref}
\usepackage{color,amsmath,amssymb,amsthm}
\usepackage{graphicx}

\newtheorem{theorem}{Theorem}[section]
\newtheorem{corollary}[theorem]{Corollary}
\newtheorem{lemma}[theorem]{Lemma}
\newtheorem{definition}[theorem]{Definition}
\newtheorem{proposition}[theorem]{Proposition}
\newtheorem{remark}[theorem]{Remark}

\title {General Least Gradient Problems with Obstacle}

\author{{Morteza Fotouhi \footnote{Department of Mathematical Sciences, Sharif University of Technology, Tehran, Iran (\href{mailto:fotouhi@sharif.edu}{\tt fotouhi@sharif.edu}).  }
\qquad Amir Moradifam\footnote{Department of Mathematics, University of California, Riverside, California, USA. E-mail: amirm@ucr.edu. Amir Moradifam is supported by NSF grant DMS-1715850. }}}

\date{}

\def\bR{\mathbb{R}}

\def\bQ{\mathbb{Q}}

\def\cA{\mathcal{A}}
\def\cH{\mathcal{H}}

\def\cC{\mathcal{C}}

\def\cK{\mathcal{K}}
\newcommand{\cL}{{\mathcal L}}
\newcommand{\ra}{\rightarrow}

\def\part{\partial}

\def\e{\epsilon}

\def\XXint#1#2#3{{\setbox0=\hbox{$#1{#2#3}{\int}$ }
\vcenter{\hbox{$#2#3$ }}\kern-.6\wd0}}

\begin{document}

\maketitle
{\small \noindent
}
\begin{abstract} 
We study existence, structure, uniqueness and regularity of solutions of the obstacle problem 
\begin{equation*}
\inf_{u\in BV_f(\Omega)}\int_{\bR^n}\phi(x,Du),
\end{equation*}
where $BV_f(\Omega)=\{u\in BV(\Omega): u\geq \psi \text{ in }\Omega\text{ and } u|_{\partial \Omega}=f|_{\partial \Omega}\}$, 
$f \in W^{1,1}_0(\mathbb{R}^n)$, $\psi$ is the obstacle, and $\phi(x,\xi)$ is a convex, continuous and homogeneous function of degree one with respect to the $\xi$ variable. We show that every minimizer of this problem is also a minimizer of the least gradient problem 
\[\inf_{u\in \cA_f(\Omega)}\int_{\bR^n}\phi(x,Du),\]
 where $\cA_f(\Omega)=\{u\in BV(\Omega): u\geq \psi, \text{ and } u=f \text{ in }\Omega^c\}$. Moreover, there exists a vector field $T$ with $\nabla \cdot T  \leq 0$ in $\Omega$ which determines the structure of all minimizers of these two problems, and  $T$ is divergence free on $\{x\in \Omega: u(x)>\psi(x)\}$ for any minimizer $u$. We also present uniqueness and regularity results that are based on maximum principles for minimal surfaces. Since minimizers of the least gradient problems with obstacle do not hit small enough obstacles,  the results presented in this paper extend several results in the literature about least gradient problems without obstacle. 

\end{abstract}
\section{Introduction}
Let $\Omega$ be a bounded open set in $\bR^n$ with Lipschitz boundary and $\phi:\Omega\times\bR^n\longrightarrow\bR$
be a continuous function satisfying the following conditions:

\begin{itemize}
\item[(C1)] There exists $\alpha>0$ such that $\alpha|\xi|\leq\phi(x,\xi)\leq\alpha^{-1}|\xi|$ for all $x\in\Omega$ and $\xi\in\bR^n$.

\item[(C2)] $\xi\mapsto\phi(x,\xi)$ is a norm for every $x$.
\end{itemize}
For our results concerning the regularity of solutions we will also assume the following three additional assumptions

\begin{itemize}

\item[(C3)] $\phi\in W^{2,\infty}_{\text{loc}}$ away from $\{\xi=0\}$, and there exists $C>0$ such that
$$\phi_{\xi_i\xi_j}(x,\xi)p^ip^j\geq C|p'|^2,$$
for all $\xi\in S^{n-1}$ and $p\in\bR^n$, where $p':=p-(p\cdot\xi)\xi$.

\item[(C4)] $\phi$ and $D_\xi\phi$ are $W^{2,\infty}$ away from $\{\xi=0\}$, and there are positive constants $\rho$ and $\lambda$ such that 
\begin{align*}
\phi(x,\xi)&+|D_\xi\phi(x,\xi)|+|D^2_\xi\phi(x,\xi)|+|D^3_\xi\phi(x,\xi)|+\rho|D_xD_\xi\phi(x,\xi)|\\
&+\rho|D_xD^2_\xi\phi(x,\xi)|+\rho^2|D_x^2D_\xi\phi(x,\xi)|\leq\lambda, \qquad\text{for all }x\in\Omega, \xi\in S^{n-1}.
\end{align*}

\item[(C5)] For the result of regularity we need to assume that the integrand $\phi(x,\xi)=\phi(\xi)$ is independent of $x$. \\

\end{itemize}

It is elementary to verify that if $\phi:\Omega\times\bR^n\longrightarrow\bR$ satisfies C1-C4, then for every $p,q \in \bR^n$ and $\lambda\in\bR$ we have
\begin{equation}\label{norm-property}
\phi_\xi(x,\lambda p)=\phi_\xi(x,p), \ \ \hbox{and} \ \ p\cdot\phi_\xi(x,p)=\phi(x,p).
\end{equation}
For  $u\in BV_{\emph{loc}}(\bR^n)$, let $\phi(x,Du)$ denote the measure defined by 
\begin{equation*}
\int_A\phi(x,Du)=\int_A\phi(x,\nu^u(x))|Du|\quad\text{for any bounded Borel set }A,
\end{equation*}
where $|Du|$ is the total variation measure associated to the vector-valued measure $Du$, and $\nu^u$ is the Radon-Nikodym derivative $\nu^u(x)=\frac{d\,Du}{d\,|Du|}$.
Basic facts about $BV$ functions imply that if $U$ is an open set, then 
\begin{equation}\label{convexity}
\int_U\phi(x,Du)=\sup\big\{\int_Uu\nabla\cdot Ydx:Y\in C_c^\infty(U;\bR^n),\;\sup\phi^0(x,Y(x))\leq1\big\},
\end{equation}
where $\phi^0(x,\cdot)$ denotes the norm on $\bR^n$ dual to $\phi(x,\cdot)$, defined by
\begin{equation*}
\phi^0(x,\xi):=\sup\{\xi\cdot p:\phi(x,p)\leq1\},
\end{equation*}
(see \cite{Amar-Bellettini, Jerrard-Moradifam-Ardian-2018}). 
For $u\in BV(\Omega)$, $\int_\Omega\phi(x,Du)$ is called the $\phi$-total variation of $u$ in $\Omega$. 
Also, if $A,\, E$ are subsets of $\bR^n$, with $A$ Borel and $E$ having finite perimeter, then we shall write $P_\phi(E;A)$ to denote the $\phi$-perimeter of $E$ in $A$, defined by
\begin{equation*}
P_\phi(E;A):=\int_A\phi(x,D\chi_E),
\end{equation*}
where $\chi_E$ is the characteristic function of $E$. We will also write $P_\phi(E)$ to denote $P_\phi(E;\bR^n)$. We shall need the following lemma. 
\begin{lemma}[Lemma 2.2 in  \cite{Jerrard-Moradifam-Ardian-2018}]\label{perimeter inequality}
Let $A\subset\bR^n$ be a Borel set and $E_1,\,E_2\subset\bR^n$ be of locally finite perimeter with respect $\phi$. Then
$$P_\phi(E_1\cup E_2;A)+P_\phi(E_1\cap E_2;A)\leq P_\phi(E_1;A)+P_\phi(E_2;A).$$
\end{lemma} 

\begin{definition}  We say that a function $u\in BV(\bR^n)$ is a $\phi$-total variation minimizing in a set $\Omega\subset\bR^n$ if 
\begin{equation*}
\int_{\bR^n}\phi(x,Du)\leq\int_{\bR^n}\phi(x,Dv)\text{ for all }v\in BV(\bR^n)\text{ such that }u=v\text{ a.e. in }\Omega^c.
\end{equation*}
Similarly, we say that $E\subset\bR^n$ of finite perimeter is $\phi$-area minimizing in $\Omega$ if 
\begin{equation*}
P_\phi(E)\leq P_\phi(F)\text{ for all }F\subset\bR^n\text{ such that }F\cap\Omega^c=E\cap\Omega^c\text{ a.e.}.
\end{equation*}
Moreover,  $E\subset\bR^n$ is called $\phi$-super (sub) area minimizing  in $\Omega$, if 
\begin{equation*}
P_\phi(E)\leq P_\phi(E\cup F)\,(\text{respectively } P_\phi(E)\leq P_\phi(E\cap F))
\end{equation*}
 for all $F\subset\bR^n$ such that $F\cap\Omega^c=E\cap\Omega^c$  almost everywhere.
\end{definition}

Let $f \in  BV(\mathbb{R}^n)$ and $\psi \in W^{1,1}(\Omega)$, and consider the obstacle least gradient problem
\begin{equation}\label{obstacle-problem-BC}
\inf_{u\in BV_f}\int_{\Omega}\phi(x,Du),
\end{equation}
where
$$BV_f(\Omega):=\{u\in BV(\Omega): u\mid_{\partial \Omega}=f \ \ \hbox{and}\ \  u(x)\geq \psi (x)   \text{ for a.e. }x\in \Omega\}.$$

Functions of least gradient was first studies by P. Sternberg, W. Graham, and W.  Ziemer in \cite{Strenberg-Williams-Ziemer}, and the results were later extended to least gradient problems with obstacle in \cite{ZZ}. Due to important applications of least gradient problems in conductivity imaging, such problems have received an extensive attention in the past decade (see \cite{Gorny, Jerrard-Moradifam-Ardian-2018, Mazon-Rossi-Segura, Moradifam, Moradifam1, MNT, MNTim, NTT07, NTT08, NTT10, Zuniga}). In this paper we will study existence and structure of minimizers, uniqueness, and regularity of minimizers of the general obstacle least gradient problem \eqref{obstacle-problem-BC}. Since minimizers of the least gradient problems with obstacle do not hit small enough obstacles,  the results in this paper extend and unify several results in the literature about least gradient problems without obstacle.

In general the problem \eqref{obstacle-problem-BC} may not have a minimizer  (see \cite{Jerrard-Moradifam-Ardian-2018}, \cite{Mazon-Rossi-Segura}, \cite{Strenberg-Williams-Ziemer}). However the relaxed problem 
\begin{equation}\label{obstacle-problem0}
\min_{u\in \cA_f} \left( \int_{\Omega}\phi(x,Du)+\int_{\partial \Omega}\varphi(x,\nu_{\Omega})|u-f| \right),
\end{equation}
always has a solution, where $\cA_f=\{u\in BV(\bR^n): u\geq \psi, \text{ and } u=f\text{ in }\Omega^c\}$, and $\nu_\Omega$ is the outer pointing unit normal vector on $\partial \Omega$. Indeed let $\{v_n\}_{n=1}^{\infty}$ be a minimizing sequence for 
\begin{equation}\label{F}
F(v):=\int_{\mathbb{R}^n}\varphi(x,dv).
\end{equation}
 Since $BV(\mathbb{R}^n)\hookrightarrow L^1_{loc}$, $F$ is coercive in $BV(\mathbb{R}^n)$ (a consequence of $C_1$) and weakly lower semicontinuous (see \cite{Jerrard-Moradifam-Ardian-2018} for more details), it follows from standard arguments that $\{v_n\}_{n=1}^{\infty}$ has a subsequence converging strongly in $L^1_{loc}$ to a function $v\in \cA_f$ with 
\[\int_{\mathbb{R}^n}\varphi(x, Du) \leq \inf_{v\in BV_f(\Omega)} \int_{\mathbb{R}^n} \varphi(x, Dv),\]
and hence $v$ is also a minimizer of \eqref{obstacle-problem0}.  However, in general, the trace $v |_{\partial \Omega}$ on $\partial \Omega$ may not be equal to $f$, leading to possible nonexistence for the problem \eqref{obstacle-problem-BC}. In addition, we shall prove the the following result.\\

\begin{remark} Since $u|_{\partial \Omega}=f$ for every $u\in BV_f(\Omega)$, the compatibility condition $f \geq \psi$ on $\partial \Omega$ must be satisfied.  Every $f\in L^1(\partial \Omega)$ can be extended to a function in $BV(\mathbb{R}^n)$ (denoted by $f$ again) with $f \geq \psi$ in $\Omega$, and throughout the paper we shall naturally assume that $f\geq \psi $ in $\bar{\Omega}$.

\end{remark}

\begin{proposition}\label{relationProp}
Let $\Omega$ be a bounded open set in $\bR^n$ with Lipschitz boundary, $f \in  BV(\mathbb{R}^n)$ and $\psi \in W^{1,1}(\Omega)$ with $f \geq \psi$ in $\Omega$, and $\phi:\Omega\times\bR^n\longrightarrow\bR$
be a continuous function satisfying C1-C2. Then \eqref{obstacle-problem0} has a solution and
\[ \inf_{u\in BV_f}\int_{\Omega}\phi(x,Du)= \min_{u\in \cA_f} \left( \int_{\Omega}\phi(x,Du)+\int_{\partial \Omega}\varphi(x,\nu_{\Omega})|u-f| \right).\]
In particular, every minimizer of \eqref{obstacle-problem-BC} is also a minimizer of \eqref{obstacle-problem0}. 
\end{proposition}

Indeed in order to prove existence of solutions to \eqref{obstacle-problem-BC} we need a condition on $\Omega$ which is defined as  follows.
\begin{definition}\label{barrier}
Let $\Omega\subset\bR^n$ be a bounded Lipschitz domain and $\phi:\Omega\times\bR^n\longrightarrow\bR$ is continuous function that satisfies C1-C2. We say that $\Omega$ satisfies the barrier condition if for $x_0\in\partial\Omega$ and $\epsilon>0$ sufficiently small, if $V$ minimizes $P_\phi(\cdot\,;\bR^n)$ in 
\begin{equation*}
\{W\subset\Omega: W\setminus B(\epsilon,x_0)=\Omega \setminus B(\epsilon,x_0)\},
\end{equation*}
then
\begin{equation*}
\partial V^{(1)}\cap\partial\Omega\cap B(\epsilon,x_0)=\varnothing.
\end{equation*}
\end{definition}

\begin{remark}
Intuitively, if  $\Omega$ satisfies the barrier condition, then at every point on $\partial \Omega$ one can decrease the perimeter of $\partial \Omega$ by pushing the boundary inwards.  
In \cite{Jerrard-Moradifam-Ardian-2018}, 
a convenient interpretation of the barrier condition, when $\partial\Omega$ is sufficiently smooth, is provided:
\begin{equation}\label{barrier-condition}
-\sum_{i=1}^n\partial_{x_i}\phi_{\xi_i}(x,Dd(x))>0,\quad\text{on a dense subset of }\partial\Omega,
\end{equation}
where $d(\cdot)$ is the signed distance to $\partial\Omega$ by 
$$d(x):=\left\{\begin{array}{ll}
\emph{dist}(x,\partial\Omega),&\text{if }x\in\Omega,\\
-\emph{dist}(x,\partial\Omega),&\text{if not.}
\end{array}\right.$$
\end{remark}

We will show that if $\Omega$ satisfies the barrier condition, then every solution of  \eqref{obstacle-problem0} is also a solution of \eqref{obstacle-problem-BC}.

\begin{theorem}\label{Existence}
Suppose that $\phi:\Omega\times\bR^n\longrightarrow\bR$ is a continuous function that satisfies C1-C2 in a bounded Lipschitz domain $\Omega\subset\bR^n$,  and $f\in C(\partial\Omega)$ with $f\geq \psi$. If $\Omega$ satisfies the barrier condition with respect to $\phi$, then every solution of \eqref{obstacle-problem0} is also a solution of \eqref{obstacle-problem-BC}. In particular, \eqref{obstacle-problem-BC} has a solution. 
\end{theorem}

We shall also prove that there exists a fixed vector field $T$ that determines the structure of level sets of the minimizers of \eqref{obstacle-problem-BC} and \eqref{obstacle-problem0}. 

\begin{theorem}\label{structure}
Let $\Omega\subset\bR^n$ be a bounded Lipschitz domain and $\phi:\Omega\times\bR^n\longrightarrow\bR$ is continuous function that satisfies C1-C2, and $f\in W^{1,1}_0(\bR^n)$. Then there exists a vector field $T\in (\cL^\infty(\Omega))^n$ with $\phi^0(x,T)\leq1$ a.e. in $\Omega$,  and $\nabla \cdot T  \leq 0$  such that 
\begin{equation}\label{first}
\phi(x,\frac{Dw}{|Dw|})=T\cdot\frac{Dw}{|Dw|},\quad |Dw|-\text{a.e. in }\Omega,
\end{equation}
\begin{equation}\label{second}
\phi(x,\nu_\Omega)|f-w|=[T,(f-w)\nu_\Omega],\quad\cH^{n-1}-\text{a.e. in }\partial\Omega \cap \{w>\psi\},
\end{equation}
for every minimizer $w$ of \ \eqref{obstacle-problem-BC} or \eqref{obstacle-problem0}. Moreover $T$ is divergence-free in $\{x\in \Omega: w(x)>\psi(x)\}$. 
\end{theorem}

The above result generalizes Theorem 1.2 in \cite{Moradifam} and simplifies to the following result in the special case $\varphi(x,\xi)=a(x)|\xi|$. 

\begin{corollary}
Let $\Omega\subset\bR^n$ be a bounded Lipschitz domain and assume that $a\in C(\bar\Omega)$ is a non-negative function, and $f\in W^{1,1}_0(\bR^n)$. 
Then there exists a vector field $T\in(\cL^\infty(\Omega))^n$ with $|T|\leq a$ a.e. in $\Omega$, and $\nabla \cdot T \leq 0$ such that every minimizer $w\in\cA_f$ of the least gradient problem 
\begin{equation}\label{Boundary1}
\inf_{v\in\cA_f}\int_\Omega a|Dv|,
\end{equation}
satisfies 
\begin{equation*}
T\cdot\frac{Dw}{|Dw|}=|T|=a,\quad |Dw|-\text{a.e. in }\Omega,
\end{equation*}
\begin{equation*}
a|f-w|=[T,(f-w)\nu_\Omega],\quad\cH^{n-1}-\text{a.e. in }\partial\Omega \cap \{w>\psi\}.
\end{equation*}
Moreover $T$ is divergence-free in $\{x\in \Omega: w(x)>\psi(x)\}$. 
\end{corollary}

The above corollary asserts that there exists a vector field $T$ such that for every minimizer $w$ of \eqref{Boundary1} the vector field $\frac{Dw}{|Dw|}$ is parallel to $T$, $|Dw|$-a.e. in $\Omega$. Moreover, if the trace of $T$ can be represented by a function $T_{tr}\in ( L^{\infty}(\partial \Omega))^n$, then up to a set with $\mathcal{H}^{n-1}$-measure zero 
\[\{x\in \partial \Omega \cap \{w>\psi \}: w|_{\partial \Omega}> f\} \subseteq \{x\in \partial \Omega: T_{tr} \cdot \nu_{\Omega}=|T_{tr}|\},\]
and similarly 
\[\{x\in \partial \Omega  \cap \{w>\psi \}: w|_{\partial \Omega}< f\} \subseteq \{x\in \partial \Omega:  T_{tr} \cdot \nu_{\Omega}=-|T_{tr}|\}.\]
In other words $w|_{\partial \Omega}=f$, $\mathcal{H}^{n-1}$-a.e. in 
\[\{ x\in \partial \Omega \cap  \{w>\psi \}: |T_{tr}\cdot \nu_{\Omega}|<|T_{tr}|\},\]
for every minimizer $w$ of \eqref{Boundary1}. These results extend the second authors results about structure of minimizers of least gradient problems \cite{Moradifam} for least gradient problems with obstacle.

We will also prove the following results about the uniqueness and regularity of minimizers of the obstacle least gradient problem \eqref{obstacle-problem-BC}.

\begin{theorem}[Comparison Principle]\label{Comparison principle}
Let $\Omega\subset\bR^n$ be a bounded Lipschitz domain with connected boundary, and assume $\phi:\Omega\times\bR^n\longrightarrow\bR$ satisfies C1-C5. Suppose that $u_1$ and $u_2$ are solutions of \eqref{obstacle-problem-BC} for boundary conditions $f_1, f_2\in C(\partial\Omega)$ respectively. Then 
\begin{equation*}
|u_1-u_2|\leq\sup_{\partial\Omega}|f_1-f_2| \qquad\text{ a.e. in }\Omega.
\end{equation*}
Moreover, 
\begin{equation}\label{boundary-monoton}
u_2\geq u_1\text{ a.e. in }\Omega, \text{ if }f_2\geq f_1\text{ on }\partial\Omega.
\end{equation}
In particular, for every $f\in C(\partial\Omega)$, there is at most one solution for \eqref{obstacle-problem-BC}. 
\end{theorem}

\begin{theorem}[Holder Regularity]\label{Holder regularity}
Suppose that $\phi:\Omega\times\bR^n\longrightarrow\bR$  satisfies C1-C5 and
let $\Omega$ be a bounded, open subset of $\bR^n$ with $C^2$ boundary which the signed distance $d(\cdot)$ to $\partial\Omega$ satisfies the relation \eqref{barrier-condition}. Assume $f\in C^{0,\alpha}(\partial\Omega)$, and $\psi\in C^{0,\alpha/2}$ for some $0<\alpha\leq1$. If $u\in BV(\Omega)$ is a solution of \eqref{obstacle-problem-BC}, then $u\in C^{0,\alpha/2}(\overline\Omega)$.\\
\end{theorem}

\begin{theorem}[Lipschitz Regularity]\label{Lipschitz regularity}
Suppose that $\phi:\Omega\times\bR^n\longrightarrow\bR$ satisfies C1-C5 and
let $\Omega$ be a bounded, open subset of $\bR^n$ with $C^2$ boundary which the signed distance $d(\cdot)$ to $\partial\Omega$ satisfies the relation \eqref{barrier-condition}. Assume $f\in C^{1,\alpha}(\partial\Omega)$, and $\psi\in C^{0,\frac{1+\alpha}2}$ for some $0<\alpha\leq1$. If $u\in  BV(\Omega)$ is a solution of \eqref{obstacle-problem-BC}, then $u\in C^{0,\frac{1+\alpha}2}(\overline\Omega)$.
\end{theorem}

\section{Structure of minimizers}\label{section-Structure of Minimizers}
In this section we study the relationship between minimizers of the least gradient problems \eqref{obstacle-problem-BC} and \eqref{obstacle-problem0}, and prove several results about existence and structure of minimizers of these problems.

Let $\nu_\Omega$ denote the outer unit normal vector to $\partial\Omega$. Then for every $V\in(\cL^\infty(\Omega))^n$ with $\nabla\cdot V\in\cL^n(\Omega)$ there exists a unique function $[V,\nu_\Omega]\in\cL^\infty_{\cH^{n-1}}(\partial\Omega)$ such that 
\begin{equation}\label{trace}
\int_{\partial\Omega}[V,\nu_\Omega]u\,d\cH^{n-1}=\int_\Omega u\nabla\cdot Vdx+\int_\Omega V\cdot D udx,
\quad u\in C^1(\bar\Omega).
\end{equation}
Moreover, for $u\in BV(\Omega)$ and $V\in(\cL^\infty(\Omega))^n$ with $\nabla \cdot V\in\cL^n(\Omega)$, the linear functional $u\mapsto(V\cdot Du)$ gives rise to a Radon measure on $\Omega$, and \eqref{trace} is valid for every $u\in BV(\Omega)$ (see \cite{Alberti, Anzellotti} for a proof).

We first show that there exists a vector field $T$ that determines the structure of all minimizers of  \eqref{obstacle-problem-BC} and \eqref{obstacle-problem0}. Next we define the dual of the least gradient problem \eqref{obstacle-problem-BC}. Let $E:(\cL^1(\Omega))^n\ra\bR$ and $G:W^{1,1}_0(\Omega)\ra\bR$ be defined as follows
\begin{equation}\label{definitionE&G}
E(P):=\int_\Omega\phi(x,P+\nabla f)dx,\quad G(u)=\left\{\begin{array}{ll}0 &u\in\cK\\+\infty &u\notin\cK,\end{array}\right.
\end{equation}
where  
\[\cK:=\{u\in W^{1,1}_0(\Omega):u\geq\psi-f\}.\]
 Then the problem \eqref{obstacle-problem-BC}  can be written as 
\begin{equation*}
(P)\qquad\inf_{u\in W^{1,1}_0(\Omega)} E(Du)+G(u).
\end{equation*}
By Fenchel duality (see Chapter III in \cite{Ekeland-Temam}) the dual problem is given by
\begin{equation*}
(P^*)\qquad\sup_{V\in (\cL^\infty(\Omega))^n}\{-E^*(V)-G^*( \nabla \cdot V)\},
\end{equation*}
where $E^*$ and $G^*$ are the Legendre-Fenchel transform of $F$ and $G$. By Lemma 2.1 in \cite{Moradifam} we have 
\begin{equation*}
E^*(V)=\left\{\begin{array}{ll}-\langle Df,V\rangle&\text{if }\phi^0(x,V(x))\leq1\text{ in }\Omega\\+\infty,&\text{otherwise.}\end{array}\right.
\end{equation*}
One can also compute $G^*:W^{-1,\infty}(\Omega)\ra\bR$ as follows.

\begin{lemma}
 Suppose $v=\nabla \cdot V$ for some $V\in(\cL^\infty(\Omega))^n$. Then
$$G^*(v)=\left\{\begin{array}{ll}< \infty,&v\in\cC^*,\\+\infty,&v\notin\cC^*,\end{array}\right.$$
where 
\[\cC^*:=\{v\in W^{-1,\infty}(\Omega):\langle v,u\rangle\leq0, \text{for all }0\leq u\in W^{1,1}_0(\Omega)\}.\]
Moreover for $v\in C^*$
\begin{equation}
G^*(v)=-\int_{\Omega} V \cdot D (\psi-f)+C(V),
\end{equation}
for some constant $C$ which only depends on $V$ near $\partial \Omega$, i.e. 
\[C(V_1)=C(V_2) \ \ \hbox{if}\ \ V_1-V_2  \in (L^{\infty}_c(\Omega))^n.\]
\end{lemma}
\begin{proof}
First note that 
\begin{equation*} 
G^*(v)=\sup_{u\in W^{1,1}_0(\Omega)}\big(\langle v,u\rangle-G(u)\big)
=\sup_{u\in \cK}\langle v,u\rangle.
\end{equation*}
Then if $v\notin\cC^*$, there exists $0\leq u_0\in W^{1,1}_0(\Omega)$ such that $\langle v,u_0\rangle>0$. Hence for any $u\in\cK$ and $\lambda>0$, we have $u+\lambda u_0\in\cK$ and $\langle v,u+\lambda u_0\rangle\ra\infty$ when $\lambda\ra\infty$.

For $v\in\cC^*$, consider the decomposition $u=u_+-u_-$ where $u_\pm=\max\{\pm u,0\}$. Then $\langle v,u\rangle\leq\langle v,-u_-\rangle$, and hence 
$$G^*(v)=\sup_{0\geq u\in\cK}\langle v,u\rangle.$$

Now consider the Lipschitz function $\eta_\epsilon\in C_0^{0,1}(\Omega)$ with value in $[0,1]$ such that $\eta_\epsilon\equiv1$ in $\Omega_\epsilon=\{x\in\Omega:\emph{dist}(x,\partial\Omega)\geq\epsilon\}$ and $\nabla\eta_\epsilon=-\frac1\epsilon\nu_\Omega$ a.e. in $\Omega\setminus\Omega_\epsilon$, in which $\nu_\Omega$ is a Lipschitz extension of the boundary normal vector of $\partial\Omega$ to its neighborhood. 
If  $\psi-f \leq u \leq 0$ in $\Omega$, we have $\eta_\epsilon(\psi-f)\in\cK$, and 
\[\eta_{\epsilon_2}(\psi-f) \geq \eta_{\epsilon_1}(\psi-f) \ \ \hbox{if} \ \ 0<\epsilon_1\leq \epsilon_2.\]
Since $v\in C^*$, $\langle v,\eta_\epsilon(\psi-f)\rangle$ is monotone in $\epsilon$ and the limit
\begin{equation}\label{limit}
\lim_{\epsilon \rightarrow 0} \langle v,\eta_\epsilon(\psi-f)\rangle
\end{equation}
exists.  Thus we have 
\begin{align*}
G^*(v)&\geq \lim_{\epsilon \rightarrow 0}\langle v,\eta_\epsilon(\psi-f)\rangle\\
&=\lim_{\epsilon \rightarrow 0} \left( \int_{\Omega\setminus\Omega_\epsilon}\frac1\epsilon (\psi-f) V \cdot \nu_{\Omega}-\eta_\epsilon V\cdot D(\psi-f)\,dx
-\int_{\Omega_\epsilon}V\cdot D(\psi-f)\,dx\right)\\
&=\lim_{\epsilon \rightarrow 0} \left( \int_{\Omega\setminus\Omega_\epsilon}\frac1\epsilon (\psi-f) V \cdot \nu_{\Omega}\right)
-\int_{\Omega}V\cdot D(\psi-f)\,dx\\
&=C(V)-\int_{\Omega}V\cdot D(\psi-f)\,dx,
\end{align*}
where
\[C(V):=\lim_{\epsilon \rightarrow 0} \left( \int_{\Omega\setminus\Omega_\epsilon}\frac1\epsilon (\psi-f) V \cdot \nu_{\Omega}\right).\]
Note that, in view of \eqref{limit}, the above limit exists and only depends on $V$ near $\partial \Omega$.

On the other hand, for every $\psi-f\leq u\leq0$, we have $0\leq\eta_\epsilon(u-(\psi-f))\in W^{1,1}_0(\Omega)$, 
so $0\geq\langle v,\eta_\epsilon(u-(\psi-f))\rangle$. 
Thus $\langle v,\eta_\epsilon u\rangle\leq\langle v,\eta_\epsilon(\psi-f)\rangle$.  
Letting $\epsilon \rightarrow 0$ we arrive at 
\[\langle v, u\rangle\leq C(V)-\int_{\Omega}V\cdot D(\psi-f)\,dx,\]
and hence
\[G^*(v) \leq C(V)-\int_{\Omega}V\cdot D(\psi-f)\,dx.\]
The proof is now complete. 
\end{proof}

\begin{proof}[Proof of Theorem \ref{structure}] The dual problem  $(P^*)$ has a solution. This follows from Theorem III.4.1 in \cite{Ekeland-Temam}. Indeed it easily follows from \eqref{convexity} that $I(v)=\int_{\Omega}\varphi (x, Dv)$ is convex, and $J: L^1(\Omega)\rightarrow \mathbb{R}$ with $J(p)=\int_{\Omega} \varphi(x, p)dx$ is continuous at $p=0$ (a consequence of $C_2$). Therefore the condition (4.8) in the statement of Theorem III.4.1 in \cite{Ekeland-Temam} is satisfied, duality gap is zero, and the dual problem $(P^*)$ has a solution. Let $T$ be a solution of the dual problem $(P^*)$, then it must satisfy $\phi^0(x,T(x))\leq1$ and $\nabla\cdot T\in\cC^*$ (i.e. $\nabla\cdot T\leq0$ in the sense of distributions). 
Moreover, we have 
\begin{equation*}
\sup(P^*)=\langle T, Df\rangle+\langle T, D (\psi-f )\rangle-C(T)=\langle T, D \psi \rangle-C(T).
\end{equation*}
Let  $w\in \cA_f$ be a  minimizer of \eqref{obstacle-problem0}, and $\epsilon>0$. Then 
\begin{align}
\int_\Omega\phi(x,Dw)&=\int_\Omega\phi(x,\frac{Dw}{|Dw|})|Dw|\geq\int_\Omega T\cdot\frac{Dw}{|Dw|}|Dw|\label{thm1-8-rel1}\\
&=\int_\Omega T\cdot Dw\notag\\
&=\sup(P^*)+\int_\Omega T\cdot D(w-\psi)+C(T)\notag\\
&=\sup(P^*)- \langle T, D(\psi-f)\rangle+C(T)+\int_\Omega T\cdot D(w-f)\notag\\
&=\sup(P^*)+G^*(\nabla\cdot T)+\int_\Omega T\cdot D(w-f)\notag\\
&= \sup(P^*)+G^*(\nabla\cdot T)+\int_\Omega T\cdot D(\eta_{\epsilon}(w-f))\notag\\
&\hspace{3cm}+\int_\Omega T\cdot D \left[(1-\eta_\epsilon)(w-f)\right] \notag
\end{align}
\begin{align}
&\geq \sup(P^*)+G^*(\nabla\cdot T)+\inf_{\psi-f\leq u\in BV_0(\Omega)}\int_\Omega T\cdot Du\label{thm1-8-rel2}\\
&\hspace{2.5cm}+\int_\Omega T\cdot D \left[(1-\eta_\epsilon)(w-f)\right]\notag\\
&= \sup(P^*)+\sup_{u\in \cK} \langle \nabla\cdot T,u\rangle+\inf_{u\in \cK}\int_\Omega T\cdot Du+\int_\Omega T\cdot D \left[(1-\eta_\epsilon)(w-f)\right]\notag\\
&= \sup(P^*)+\sup_{u\in \cK} \langle \nabla\cdot T,u\rangle-\sup_{u\in \cK}\langle \nabla\cdot T,u\rangle+\int_\Omega T\cdot D \left[(1-\eta_\epsilon)(w-f)\right]\notag\\
&= \sup(P^*)+\int_\Omega T\cdot D \left[(1-\eta_\epsilon)(w-f)\right]\notag\\
&=  \sup(P^*)+\int_\Omega T\cdot\big[ (w-f)D (1-\eta_\epsilon) +(1-\eta_\epsilon)D(w-f)\big]\notag\\
&\geq \sup(P^*)-\int_{\Omega \setminus \Omega_\epsilon}\phi(x, \frac{\nu_{\Omega}}{\epsilon}(w-f))+\int_\Omega (1-\eta_\epsilon)T\cdot D(w-f)\label{thm1-8-rel3}\\
&= \sup(P^*)-\int_{\Omega \setminus \Omega_\epsilon}\phi(x, \frac{\nu_{\Omega}}{\epsilon})|w-f|-\|T\|_{(\cL^\infty(\Omega))^n}\int_{\Omega \setminus \Omega_\epsilon} \big| D(w-f)\big|.  \notag
\end{align}
Letting $\epsilon \rightarrow 0$, we have $\int_{\Omega \setminus \Omega_\epsilon} \big| D(w-f)\big|\ra0$ and get 
\[\int_{\Omega} \phi(x,Dw)+\int_{\partial\Omega} \phi(x,\nu_{\Omega})|w-f|\geq \sup(P^*)=\inf(P).\]
On the other hand since $BV_f(\Omega) \subset \cA_f$, the above inequality also holds in the opposite direction. Thus
\begin{equation}\label{this}
\inf_{w\in \cA_f} \left( \int_\Omega\phi(x,Dw)+\int_{\partial\Omega} \phi(x,\nu_{\Omega})|w-f| \right)=\inf_{w\in BV_f(\Omega)}\int_\Omega\phi(x,Dw).
\end{equation}
Note also that if $w\in \cA_f$ is a minimizer of \eqref{obstacle-problem0}, then all the above inequalities are equalities. In particular \eqref{first} and \eqref{second} hold because of \eqref{thm1-8-rel1} and \eqref{thm1-8-rel3}, and we can deduce by \eqref{thm1-8-rel2} that
\begin{equation}\label{divergenceIsFree!}
\inf_{\psi-f\leq u\in BV_0(\Omega)}\int_\Omega T\cdot Du
=\lim_{\epsilon \rightarrow 0}\int_\Omega T\cdot D(\eta_\epsilon(w-f)).
\end{equation}
Now let $\omega \Subset \Omega$ and suppose $w>\psi$ on $\omega$. Then for $\varphi \in C^{\infty}_c(\omega)$ and $|t|$ small, we have $w+t \varphi>\psi$ in $\omega$. Hence for $\epsilon$ small enough 
\[w+t\varphi-f=\eta_{\epsilon}(w+t\varphi -f) \ \ \hbox{in}\ \ \omega,\]
and $\psi-f\leq\eta_{\epsilon}(w+t\varphi -f) \in BV_0(\Omega)$. Thus it follows from \eqref{divergenceIsFree!} that 
$$\lim_{\epsilon \rightarrow 0}\int_\Omega T\cdot D(\eta_\epsilon(w-f))\leq\int_\Omega T\cdot D(\eta_\epsilon(w+t\varphi-f)),$$
then
\[\lim_{\epsilon \rightarrow 0}\int_\Omega T\cdot D(t\eta_\epsilon\varphi)\geq0 , \ \ \forall t\in (-\delta,\delta),\]
for some $\delta>0$. Therefore 
\[\langle \nabla \cdot T, \varphi \rangle= 0, \ \ \forall \varphi \in C^{\infty}_0(\omega),\]
and consequently $T\in (L^{\infty}(\Omega))^n$ is divergence-free on $\{w>\psi\}$.  
\end{proof}
\vspace{.3cm}
\noindent 
{\bf Proof of Proposition \ref{relationProp}.} The proof follows from \eqref{this} in the proof of Theorem \ref{structure}, and the argument right before the statement of Proposition \ref{relationProp}. \hfill $\square$

\section{Existence}
In this section we study the existence of the obstacle least gradient problem \eqref{obstacle-problem-BC}, and prove Theorem \ref{Existence}. Consider an arbitrary function $u\in\cA_f$ and let
\begin{align*}
E_t:=&\{x\in\bR^n:  u(x)> t\},\\
L_t:=&\{x\in\bR^n:  f(x)> t\},\\ 
O_t:=&\{x\in\bR^n:  \psi(x)> t\}.
\end{align*}
The following theorem shows that the level sets of the solutions of \eqref{obstacle-problem0} satisfy in an obstacle  $\phi$-area minimizing problem.
\begin{theorem}\label{minimal-superset}
Let $\Omega$ be a bounded Lipschitz domain and $u$ be a solution of \eqref{obstacle-problem0}, then $E_t$ is a solution of the following variational problem,
\begin{equation}\label{variational problem}
\min\{P_\phi(E;\Omega): E\cap\Omega^c=L\cap\Omega^c \text{ and } E\supset O\cap\Omega\},
\end{equation}
in which $O=O_t$ and $L=L_t$.
\end{theorem}

\begin{remark}
It is not difficult to see that $\partial E_t\setminus\bar O_t$ is locally $\phi$-minimizing in $\Omega$ as well as $\partial E_t\cap\bar O_t$ is locally $\phi$-super minimizing in $\Omega$. 
\end{remark}

In order to prove Theorem \ref{minimal-superset}, we need the following lemma. It will also 
help us to study the relation between the minimizers of \eqref{obstacle-problem-BC} and \eqref{obstacle-problem0}.
Therein $v^+$ and $v^-$ stand for  the outer and inner trace of $v\in BV(\bR^n)$ on $\partial\Omega$. 

\begin{lemma}\label{stability of solution}
Assume $u_k$ is a solution of  \eqref{obstacle-problem0} for the obstacle $\psi_k$ such that $\psi_k\nearrow\psi$  and
$$u_k\longrightarrow u\text{ in }\cL^1(\Omega)\quad\text{ and }\quad u_k^\pm\longrightarrow u^\pm\text{ in }\cL^1(\partial\Omega).$$
Then $u$ is a solution of  \eqref{obstacle-problem0}  for the obstacle $\psi$.
\end{lemma}
\begin{proof} The proof is similar to the proof of Lemma 2.7 in \cite{Jerrard-Moradifam-Ardian-2018} and we present it here for the sake of completeness. Given $g\in L^1(\partial\Omega;\cH^{n-1})$, define
\begin{equation*}
I_\phi(v;\Omega,g):=\int_{\partial\Omega}\phi(x,\nu_\Omega)|g-v^-|\,d\cH^{n-1}+\int_\Omega\phi(x,Dv),
\end{equation*}
where $\nu_\Omega$ denotes the outer unit normal to $\Omega$.
From the upper semicontinuity of the $\phi$-total variation 
$$\int_\Omega\phi(x,Du)\leq\liminf_k\int_\Omega\phi(x,Du_k),$$
and the $\cL^1$ convergence of the trace, implies that
\begin{equation}\label{semicontinuity of operator I}
I_\phi(u;\Omega,u^+)\leq\liminf_k I_\phi(u_k;\Omega,u^+_k).
\end{equation}
Now for any $v\in BV(\bR^n)$ such that $v\geq\psi$, then $v\geq\psi_k$ and we have 
\begin{align*}
I_\phi(u_k;\Omega,u_k^+)&\leq I_\phi(v;\Omega,u_k^+)\\
&\leq I_\phi(v;\Omega,u^+)+\int_{\partial\Omega}\phi(x,\nu_\Omega)|u^+-u_k^+|\,d\cH^{n-1}\\
&\leq I_\phi(v;\Omega,u^+)+\alpha^{-1}\int_{\partial\Omega}|u^+-u_k^+|\,d\cH^{n-1}.
\end{align*}
It follows from this and \eqref{semicontinuity of operator I} that $I_\phi(u;\Omega,u^+)\leq I_\phi(v;\Omega,u^+)$.
\end{proof}
\begin{proof}[Proof of Theorem \ref{minimal-superset}]
For $t\in\bR$, let $u_1:=\max(u,t)$, $u_2:=u-u_1$, $\psi_1:=\max(\psi,t)$. Consider $v\in BV(\bR^n)$ such that $v=u_1$ a.e. in $\Omega^c$ and $\psi_1\leq v$, then $\psi\leq\psi_1+u_2\leq v+u_2$ and $v+u_2=u$ a.e. in $\Omega^c$, where we have used the assumption $\psi\leq u$. Since $u$ is a solution of \eqref{obstacle-problem0}, we can write
\begin{align*}
\int_\Omega\phi(x,Du_1)+\int_\Omega\phi(x,Du_2)=&\int_\Omega\phi(x,Du)\\
\leq&\int_\Omega\phi(x,D(v+u_2))\\
\leq&\int_\Omega\phi(x,Dv)+\int_\Omega\phi(x,Du_2).
\end{align*}
Hence $u_1$ is also a solution of \eqref{obstacle-problem0} for the obstacle $\psi_1$ and the boundary condition $f_1:=\max(f,t)$. Repeating the same argument, one verifies that
\begin{equation*}
\chi_{\epsilon,t}:=\min(1,\frac1\epsilon u_1)=\begin{cases}
0&\text{if }u\leq t,\\
\epsilon^{-1}(u-t)&\text{if }t\leq u\leq t+\epsilon,\\
1&\text{if }t+\epsilon\leq u,
\end{cases}
\end{equation*}
is also a solution of \eqref{obstacle-problem0} for the obstacle $\psi_{\epsilon,t}:=\min(1,\frac1\epsilon\psi_1)$, and boundary condition $f_{\epsilon,t}:=\min(1,\frac1\epsilon f_1)$. 

It is straightforward to check that 
$$\chi_{\epsilon,t}\rightarrow\chi_t:=\chi_{E_t}\text{ in }\cL^1_{\text{loc}}(\bR^n), \qquad \chi^\pm_{\epsilon,t}\rightarrow\chi^\pm\text{ in }\cL^1(\partial\Omega;\cH^{n-1}).$$
Notice that $\psi_{\epsilon,t}\nearrow\chi_{O_t}$. Thus Lemma \ref{stability of solution} implies that $\chi_{E_t}$ is a solution of \eqref{obstacle-problem0} for the obstacle $\chi_{O_t}$ and the boundary condition $\chi_{L_t}$. 

\end{proof}

Next we can use the barrier condition to prove the following lemma proof of which is similar to Lemma 3.4 in \cite{Jerrard-Moradifam-Ardian-2018} and we omit it. Remind that for a measurable subset $E$ of $\bR^n$, we define  
\begin{equation*}
E^{(1)}:=\{x\in\bR^n:\lim_{r\rightarrow0}\frac{\cH^n(B(r,x)\cap E)}{\cH^n(B(r))}=1\}.
\end{equation*}

\begin{lemma}\label{boundary of superset}
Let $\Omega$ be a bounded Lipschitz domain satisfying the barrier condition with respect to $\phi$, and assume that 
$E$ is a solution of \eqref{variational problem}. Then
$$\{x\in\part\Omega\cap\part E^{(1)}: B(\epsilon,x)\cap \part E^{(1)}\subset\bar\Omega\;\text{for some }\epsilon>0\}\subset\bar O.$$
\end{lemma}

\begin{proof}[Proof of Theorem \ref{Existence}]
The proof follows from Proposition \ref{relationProp}, Lemma \ref{boundary of superset}, Theorem \ref{minimal-superset}, and an argument similar to that of Theorem 1.1 in \cite{Jerrard-Moradifam-Ardian-2018}. 
\end{proof}

\section{Maximum and comparison principles}

This section is devoted maximum and comparison principles which will be our main tools in proving uniqueness and regularity results. At the first, we review some well-known definition and results about the regularity theory for minimal surfaces.


\begin{definition}
Let $E\subset\bR^n$. A point $x\in\partial E$ is called a regular point if there exists $\rho>0$ such that $\partial E\cap B(x,\rho)$ is a $C^2$ hypersurface. We denote the set of all regular points of $\partial E$ by $\emph{reg}(\partial E)$. We say that $x$ is a singular point if $x \in \emph{sing}(\partial E)=\partial E\setminus\emph{reg}(\partial E)$. 
\end{definition}

The following estimate on the size of singular sets of $\phi$-area minimizing sets has been proved in \cite{Schoen-Simon-Almgren-1977}, (see also Remarks 2.7 and 2.8 in  \cite{Jerrard-Moradifam-Ardian-2018}).
\begin{theorem}\label{Hausdorff dimension  of singularity}
Let $\Omega\subset\bR^n$, and assume $\phi:\Omega\times\bR^n\longrightarrow\bR$ satisfies C1-C4. 
If $E$ is $\phi$-area minimizing in $\Omega$, then
\begin{equation*}
\left\{\begin{array}{ll}
\cH^{n-3}(\emph{sing}(\partial E^{(1)})\cap\Omega)<\infty,&\text{if }n\geq4,\\
\emph{sing}(\partial E^{(1)})\cap\Omega=\varnothing,&\text{if }n\leq3.
\end{array}\right.
\end{equation*}
\end{theorem}

We shall also need the following proposition which states that every connected components of regular points of  a $\phi$-area minimizing set $E$ in $\Omega$ must reach the boundary $\partial \Omega$. 

\begin{proposition}\label{component-regular-points}
Let $\Omega$ be a bounded Lipschitz domain with connected boundary and assume that $E\subset\bR^n$ is a solution of 
\eqref{variational problem} for some sets $(L,O)$. If $R$ is a nonemtpy connected component of $\emph{reg}(\partial E^{(1)})\cap\Omega$, then $\bar R\cap\partial\Omega\neq\varnothing$ or $\bar R\cap\bar O\neq\varnothing$.
\end{proposition}
\begin{proof}
The proof follows directly from Lemma 4.2 in \cite{Jerrard-Moradifam-Ardian-2018}. In fact, if $\bar R\cap\bar O=\varnothing$, it will be a $\phi$-area minimizer and we can apply that lemma.
\end{proof}

In order to prove the strict maximum principle, we first prove a couple of intermediate results.  

\begin{lemma}\label{sub-to-sol}
Assume that $\phi$ satisfies  conditions C1-C2. 
Let $E$ be a $\phi$-sub (or  $\phi$-super) area minimizing in $\Omega$. There exists a $\phi$-area minimizing $G$ such that $G\cap\Omega^c=E\cap\Omega^c$ as well as $G\supseteq E$ (or $G\subseteq E$).
\end{lemma}
\begin{proof} First note that there is a $\phi$-area minimizing set $G$ in $\Omega$ such that $G\cap\Omega^c=E\cap\Omega^c$. 
Since $E$ is $\phi$-sub area minimizing, 
\[P_\phi(E) \leq P_\phi(E\cap G).\]
Thus it follows from Lemma \ref{perimeter inequality} that 
\[P_\phi(E\cup G) \leq P_\phi(G).\]
Hence $\tilde{G}=E \cup G$ is also $\phi$-area minimizing  and $E\subseteq \tilde{G}$. One can similarly show that every $\phi$-super area minimizing set contains a $\phi$-area minimizing set $G$ with the stated properties. 
\end{proof}

We will deduce the uniqueness of the solution and the comparison principle (Theorem \ref{Comparison principle})  from the following theorem. 

\begin{theorem}\label{comparison-sets}
Assume that $\phi$ satisfies conditions C1-C5.
Suppose that $E_1$ and $E_2$ are solutions of \eqref{variational problem} respectively for pairs of sets $(L_1,O_1)$ and $(L_2,O_2)$. Also, we have 
\begin{equation*}
L_1\Subset L_2\text{ and }O_1\Subset O_2.
\end{equation*}
Suppose $\Omega$ satisfies the barrier condition, or 
\begin{equation}\label{thm-comp-relation}
\partial E_1^{(1)}\setminus E_2^{(1)}\subset\Omega\text{ and }\partial E_2^{(1)}\cap \overline E_1^{(1)}\subset\Omega,
\end{equation}
then $E_1^{(1)}\Subset E_2^{(1)}$.
\end{theorem}
\begin{proof}In view of Theorem \ref{Hausdorff dimension  of singularity}, $\textsl{int}(E_i^{(1)})$ differs from $E_i$ in a set of measure zero and we replace $E_i$ by $\textsl{int}(E_i^{(1)})$. We prove the result in a series of steps. 

{\it Step 1.} We will show that  $G=E_1\cap E_2$ and $F=E_1\cup E_2$ are solutions of \eqref{variational problem} for the pairs of sets $(L_1,O_1)$ and $(L_2,O_2)$, respectively. Since $E_1$ and $E_2$ are solutions of \eqref{variational problem},
$$P_\phi(E_1)\leq P_\phi(G),\quad\text{ and }P_\phi(E_2)\leq P_\phi(F).$$
 By Lemma \ref{perimeter inequality}, we have 
$$P_\phi(G)+ P_\phi(F)\leq P_\phi(E_1)+P_\phi(E_2),$$
and hence $P_\phi(G)=P_\phi(E_1)$ and $P_\phi(F)=P_\phi(E_2)$. Thus $G$ and $F$ are also solutions of the problem \eqref{variational problem}. 

{\it Step 2.} If  $x_0\in\partial E_1\cap\partial F$, then there is a neighborhood of $x_0$ in which $E_1$ is a $\phi$-area minimizing and $F$ is $\phi$-super area minimizing. This  immediately follows from the observation that  $x_0\notin \bar O_1\cup \partial\Omega$. Notice that $x_0\in \partial \Omega \cap \partial E_1\cap\partial F$ violates the barrier condition. 
 
{\it Step 3:} In this step we show that if $\partial E_1\cap\partial F\neq\varnothing$, then \hbox{$\cH^{n-2}(\partial E_\nu\cap\partial F)>0$}, where $E_\nu=E_1+\nu$ for some small vector $\nu \in \mathbb{R}^n$. In order to see this, define 
$$\Omega_\delta=\{x\in\Omega: \emph{dist}(x,\partial\Omega)>\delta\},$$
and choose $\delta>0$ such that 
\begin{equation*}
\emph{dist}(\partial E_1\cap\Omega_\delta^c,\partial F\cap\Omega_\delta^c)>\delta,\ \ \emph{dist}(\bar O_1,O_2^c)>\delta.
\end{equation*}
Let $x_0\in \partial E_1\cap\partial F$ and choose $y\in B(x_0, \delta )\cap F^c$.  Set  $\nu:=y-x_0$ and $E_\nu=E_1+\nu$. By (C5), $\phi(x,\xi)=\phi(\xi)$ and hence $E_\nu$ is also a solution of \eqref{variational problem} for the pair of sets $(L_1+\nu,O_1+\nu)$ in $\Omega_\delta$. Then it follows from an argument similar to the one used in the proof of Theorem 4.6 in \cite{Jerrard-Moradifam-Ardian-2018} that  
\begin{equation}\label{intersectionIs Large}
\cH^{n-2}(\partial E_\nu\cap\partial F)>0.
\end{equation}
As in step 1, replace $F$ by $F\cup E_\nu$. 

{\it Step 4:} In view of Theorem \ref{Hausdorff dimension  of singularity} and \eqref{intersectionIs Large}, there exists a regular point $x_1$ of $\partial E_\nu$ such that $x_1\in \partial E_\nu\cap\partial F$and $x_1$ is a Lebesgue point of $\partial E_\nu\cap\partial F$ with respect to the measure $\cH^{n-2}$. In this step, we will show that  there is a neighborhood of $x_1$ in $\partial E_\nu$ that is a subset of $\partial E_\nu\cap\partial F$. 
Consider a ball  $B=B_r(x_1)$ such that $E_\nu\cap B$ is a $C^2$ hypersurface, and towards a contradiction assume that $E_\nu\cap\partial B\neq F\cap \partial B$. According to Lemma \ref{sub-to-sol}, there is a $\phi$-area minimizing $G$, such that $G\subseteq F$ and $G\cap B^c=F\cap B^c$. 
Notice that $\cH^{n-2}(\partial E_\nu\cap\partial G\cap B)>0$, since either $\partial G$ intersects $\partial E_\nu$ transversally or contains $\partial E_\nu\cap\partial F$. 

Now repeat Step 1 to find two $\phi$-area minimizing $E_\nu\cup G$ and $E_\nu\cap G$, which intersects in a set with positive $\cH^{n-2}$-measure. 
Then by Theorem \ref{Hausdorff dimension  of singularity} there is a point $x_*$ such that $E_\nu\cup G$ and $E_\nu\cap G$ are regular at that. By Lemma 4.4 in \cite{Jerrard-Moradifam-Ardian-2018} we conclude that $\partial(E_\nu\cup G)=\partial(E_\nu\cap G)$ in a neighborhood of $x_*$. This yields that $E_\nu=G$ in an open subset of $\partial E_\nu\cap B$. The boundary of this set has  positive $\cH^{n-2}$-measure, and we can repeat the above argument to prove that $E_\nu\cap B=G\cap B$ (see the proof of Theorem 4.6 in \cite{Jerrard-Moradifam-Ardian-2018} for more details). Therefore, $E_\nu\cap\partial  B= G\cap\partial B=F\cap\partial B$. This is a contradiction, and hence $\partial E_\nu$ is a subset of $\partial E_{\nu} \cap \partial F $ in a neighborhood of $x_1$.

{\it Step 5:} In this step we show that  $E_1\Subset (E_1\cup E_2)^{(1)}$.  Towards a contradiction suppose this is not the case. Then by steps 3 and 4, we know that each connected component of $ \partial E_\nu\cap\partial F$ is an open subset of  $\partial E_\nu$ for some $\nu \in \mathbb{R}^n$. It follows from Proposition \ref{component-regular-points} that $ \partial E_\nu\cap\partial F$ intersects the boundary  $\partial \Omega$ or the obstacle $O_1+\nu$, which contradicts the assumptions of the theorem, and hence  $E_1\Subset (E_1\cup E_2)^{(1)}$. 

{\it Step 6:} Finally we prove that $E_1\Subset E_2$. 
First we will show that $E_1\subset E_2$, toward a contradiction assume that $E_1\setminus E_2$ has nonempty interior. 
Since $E_1\Subset F=(E_1\cup E_2)^{(1)}$, then we have $\partial F\subseteq\partial E_2$.
On the other hand, from topological point of view 
\begin{equation}\label{boundary-inclusion}
\partial E_2\subseteq\partial F\cup\partial(E_1\setminus E_2).
\end{equation}
If there exists some point $x_0\in\partial(E_1\setminus E_2)\setminus\partial E_2$, then we  must have
$$x_0\in\textsl{int}(E_2^c)\cap \partial E_1\subset\textsl{int}(E_2^c)\cap F\subseteq E_1,$$
which contradicts $x_0\in\partial E_1$ ($E_1$ is open). 
It yields that   $\partial(E_1\setminus E_2)\subset\partial E_2$. 
Therefore, $\partial E_2=\partial F\cup\partial(E_1\setminus E_2)$ by \eqref{boundary-inclusion}, which means that the perimeter of $F$ is less than the perimeter of $E_2$ unless $\cH^{n-1} (\partial(E_1\setminus E_2))=0$. This contradicts the assumption $\textsl{int}(E_1\setminus E_2)\neq\varnothing$. 
Hence  $E_1 \cup E_2$, and consequently $E_1 \Subset F=E_2$ by the the conclusion in Step 5.
\end{proof}

\begin{remark}
When $n=2$ or $3$, the statement in Theorem \ref{comparison-sets} holds without condition (C5). 
Because all $\phi$-area minimizing sets are regular even $\phi$ depends on variable $x$ (Theorem \ref{Hausdorff dimension  of singularity}).
Hence we does not need steps 3, and in step 4 we can choose $\nu=0$. A similar argument implies $E_1 \Subset E_2$. 
\end{remark}

\begin{proof}[Proof of Theorem \ref{Comparison principle}]
The proof is inspired by Theorem 1.4 from \cite{Jerrard-Moradifam-Ardian-2018}. Suppose that \eqref{boundary-monoton} is not true. Since
$$\{x\in\Omega: u_1(x)>u_2(x)\}=\bigcup_{(\lambda_1,\lambda_2)\in\bQ\times\bQ}\{x\in\Omega:u_1(x)>\lambda_1>\lambda_2\geq u_2(x)\},$$
there must be some rational numbers $\lambda_1>\lambda_2$ such that 
$$\cH^n(\{x\in\Omega:u_1(x)>\lambda_1>\lambda_2\geq u_2(x)\})>0.$$
Now define 
$$E_i:=\{x\in\bR^n:u_i(x)>\lambda_i\},$$
then we have $\cH^n(E_1\setminus E_2)>0$. 
On the other hand, we can easily verify that the conditions of Theorem \ref{comparison-sets} are satisfies, and hence $E_1^{(1)}\Subset E_2^{(1)}$. 
\end{proof}

The idea in the proof of Theorem \ref{comparison-sets} allow us to prove a strict maximum principle for $\phi$-sub and super area minimizing sets. This result generalizes 
the result in \cite{Simon-1987} and \cite{Zuniga}.

\begin{theorem}[Strict maximum principle]\label{maximum-principle}
Assume that $\phi$ satisfies the conditions C1-C5.
Let $E \subset \mathbb{R}^n$ be $\phi$-sub area minimizing and $F \subset \mathbb{R}^n$ be $\phi$-super area minimizing relative to an open set $\Omega$, and
\[E \setminus \Omega \Subset F \setminus \Omega.\]
Suppose $\Omega$ satisfies the barrier condition, 
then
\[E^{(1)} \Subset F^{(1)}.\]
\end{theorem}

\begin{proof}
By Lemma \ref{sub-to-sol}, there exists $\phi$-area minimizing sets $\tilde{E}$ and $\tilde{F}$ such that $\tilde E\supseteq E$ and $\tilde F\subseteq F$. Since $\Omega$ satisfies the barrier condition, 
\[ \partial \tilde{E}^{(1)} \setminus  \tilde{F}^{(1)} \subset \Omega \ \ \hbox{and}\ \ \partial \tilde{F}^{(1)}\cap \overline{\tilde{E}}^{(1)} \subset \Omega.  \]
By Theorem 4.6 in \cite{Jerrard-Moradifam-Ardian-2018} we have $\tilde E^{(1)} \subset \tilde F^{(1)}$. Moreover $\tilde E^{(1)} \Subset \tilde F^{(1)}$ if $n \leq 3$. 
In order to prove the theorem for $n\geq 4$, note that since  $E\cap\Omega^c\Subset F\cap\Omega^c$, there is a $\delta>0$ such that 
\begin{equation*}
\emph{dist}(\partial E\cap\Omega_\delta^c,\partial F\cap\Omega_\delta^c)>\delta.
\end{equation*}
Let $x_0\in \partial \tilde E^{(1)}\cap\partial \tilde F^{(1)}$ and choose $y\in B(x_0, \delta )\cap \tilde F^c$.  Set  $\nu:=y-x_0$ and $E_\nu=\tilde E+\nu$. Since we have assumed (C5), $\phi(x,\xi)=\phi(\xi)$, and hence $E_\nu$ is also a $\phi$-area minimizer in $\Omega_\delta$. Observe that 
\[ \partial E_{\nu}^{(1)} \setminus \tilde F^{(1)} \subset \Omega_\delta \ \ \hbox{and}\ \ \partial \tilde F^{(1)}\cap \overline{E}_{\nu}^{(1)} \subset \Omega_\delta.  \]
 It again follows from Theorem 4.6 in \cite{Jerrard-Moradifam-Ardian-2018} that $E_{\nu}^{(1)} \subset \tilde F^{(1)}$ which is a contradiction. Thus $\partial \tilde E^{(1)}\cap\partial \tilde F^{(1)}=\varnothing$, and the proof is complete. 
\end{proof}

We shall need the following proposition to prove regularity results for solutions of  \eqref{obstacle-problem-BC}. 

\begin{proposition}\label{minimizer of distance}
Under the assumption of Theorem \ref{comparison-sets},
if \hbox{$d=\textsl{dist}(\partial E_1\cap\Omega,\partial E_2\cap\Omega)$} and this distance is taken in points $|x-y|=d$, such that $x\in\partial E_1\cap\Omega$ and $y\in \partial E_2\cap\Omega$, then  either $x\in \bar O_1\cup\partial\Omega$ or $y\in\partial\Omega$.
\end{proposition}
\begin{proof}
Consider  the points $x$ and $y$ such that violate the statement.  
Let $\nu=y-x$,  the translation $\tilde E_1=\nu+E_1$ remains a solution of \eqref{variational problem}  for the pair of sets $(\nu+L_1,\nu+O_1)=:(\tilde L_1,\tilde O_1)$ in $\tilde \Omega:=\nu+\Omega$. 
According to our assumption $\tilde L_1\Subset L_2$ and $\tilde O_1\cap\partial E_2=\varnothing$.
Choose $\epsilon>0$ such that $\tilde O_1+B_\epsilon\Subset E_2$, and define $\tilde O_2:=O_2\cup(\tilde O_1+B_\epsilon)$ which satisfies $\tilde O_2\Supset \tilde O_1$. Then $E_2$ is also a solution for $(L_2,\tilde O_2)$. 
On the other hand, $y\in \partial \tilde E_1\cap\partial E_2$ and this contradicts Theorem \ref{comparison-sets}, for $\tilde E_1$ and $E_2$ in the domain $\Omega\cap\tilde\Omega$.
\end{proof}

\section{Regularity of solutions}

First of all we shall notice that the continuity of the solution of   \eqref{obstacle-problem-BC} is a straightforward result of the geometric comparison principle, Theorem \ref{comparison-sets}. The proof is similar to Theorem 1.3 in \cite{Jerrard-Moradifam-Ardian-2018}, then we  just give the statement without proof in the following proposition. 

\begin{proposition}[Continuity]\label{Continuity}
Let $\Omega\subset\bR^n$ be a bounded Lipschitz domain with connected boundary, and assume $\phi:\Omega\times\bR^n\longrightarrow\bR$ satisfies C1-C5. 
If  $u$ is a solution of  \eqref{obstacle-problem-BC}, then $u$ is continuous.
\end{proposition}



In order to study the Holder regularity, we need the following property for the norm $\phi(x,\xi)$.
\begin{lemma}\label{norm-property2}
If $\phi:\Omega\times\bR^n\longrightarrow\bR$ satisfies C1-C4, then for every $p$ and $q$ we have
\begin{equation*}
p\cdot\phi_\xi(x,q)\leq\phi(x,p).
\end{equation*}
\end{lemma}
\begin{proof}
By the norm property \eqref{norm-property}, we can assume $\phi(x,p)=\phi(x,q)=1$.
Let $f(t):=\phi(x,tp+(1-t)q)$, we have $f(0)=1$ and for $0<t<1$
$$f(t)\leq t\phi(x,p)+(1-t)\phi(x,q)=1.$$
Alos, for $t<0$ we have
$$f(t)\geq \phi(x,(1-t)q)-\phi(x,tp)=(1-t)-|t|=1.$$
Thus $f'(0)\leq0$ which yields 
$$\phi_\xi(x,q)\cdot(p-q)\leq0.$$
Using the norm property \eqref{norm-property}, $q\cdot\phi_\xi(x,q)=\phi(x,q)=1$ to deduce the lemma.
\end{proof}

Now we are going to construct barriers and prove a comparison principle for such barriers. The results and the proofs in this section are inspired by \cite{ZZ}. 
\begin{lemma}\label{barriers}
Let $\Omega\subset\bR^n$ be a bounded Lipschitz domain. Suppose $u\in C^0(\bar\Omega)\cap BV(\Omega)$ is a solution of \eqref{obstacle-problem-BC} and $v\in C^2(\Omega)\cap C^0(\bar\Omega)$ satisfies
\begin{itemize}
\item[(i)] $|\nabla v|>0$ in $\Omega$,
\item[(ii)] $u\geq v$ on $\partial\Omega$,
\item[(iii)] $\cL v>0$ in $\Omega$,
\end{itemize}
where $\cL v=\sum_{i=1}^n\partial_{x_i}\phi_{\xi_i}(x,Dv(x))$. Then $u\geq v$ in $\Omega$. Similarly, if inequalities $(ii)$ and $(iii)$ are reserved, then $u\leq v$ in $\Omega$.
\end{lemma}
\begin{proof}
Let $E=\{x\in\Omega:v(x)>u(x)+\e\}$ for some $\e>0$, and $w=\max(u,v-\e)$. Notice that $w\in BV(\Omega)\cap C^0(\bar\Omega)$, $w=u$ on $\partial\Omega$ and $w\geq\psi$.  
Now let $\eta\in C_0^\infty(\Omega)$ satisfy $\eta=1$ on $E$ and $0\leq\eta\leq1$ in $\Omega$. 
Set
$$g=\eta\phi_\xi(x,Dv),$$
so that $g\in[C_0^1(\Omega)]^n$. By Theorem 2.1 in \cite{Amar-Bellettini}, we can write  
\begin{align*}
\int_E(u-w)\nabla\cdot g\,dx&=\int_\Omega(u-w)\nabla\cdot g\,dx=-\int_\Omega g\cdot D(u-w)\,dx\\
&=-\int_E g\cdot D(u-w)=-\int_E g\cdot Du+\int_E g\cdot Dv\\
&=-\int_E \phi_\xi(x,Dv)\cdot \frac{Du}{|Du|}|Du|+\int_EDv\cdot\phi_\xi(x,Dv)\,dx\\
&\geq-\int_E\phi(x,\frac{Du}{|Du|})|Du|+\int_E\phi(x,Dv),
\end{align*}
where in the last line we use the norm properties in Lemma \ref{norm-property2} and relation \eqref{norm-property}. 
Since $u-w<0$ and $\nabla\cdot g>0$ in $E$ (condition $(iii)$), we have 
$$\int_E\phi(x,Du)=-\int_E\phi(x,\frac{Du}{|Du|})|Du|>\int_E\phi(x,Dw),$$
which violates the fact that $u$ is a minimal solution of \eqref{obstacle-problem-BC}.
\end{proof}

Here, we prove firstly the regularity of the solutions near the boundary.
\begin{lemma}\label{boundary-regularity}
Suppose $\Omega$ is a bounded, open subset of $\bR^n$ with $C^2$ boundary which the signed distance $d(\cdot)$ to $\partial\Omega$ satisfies the relation \eqref{barrier-condition}. Assume $f\in C^{0,\alpha}(\partial\Omega)$, and $\psi\in C^{0,\alpha/2}$ for some $0<\alpha\leq1$. 
If $u\in C^0(\overline\Omega)\cap BV(\Omega)$ is a solution of \eqref{obstacle-problem-BC}, then there exists positive constants $\delta$ and $C$ depending only on $\|f\|_{C^{0,\alpha}(\partial\Omega)}$, $\|\psi\|_{C^{0,\alpha/2}}$ and $\|u\|_{C^0(\bar\Omega)}$ such that 
\begin{equation*}
|u(x)-u(x_0)|\leq C|x-x_0|^{\alpha/2},
\end{equation*}
whenever $x_0\in\partial\Omega$ and $x\in\bar\Omega$ with $|x-x_0|<\delta$.
\end{lemma}
\begin{proof}
For each $x_0\in\partial\Omega$ we will construct functions $w^+,w^-\in C^2(U)\cap C^0(\bar U)$ where $U=B(x_0,\delta)\cap\Omega$ for some $\delta>0$  is to be determined later, such that
\begin{itemize}
\item[(i)] $w^+(x_0)=w^-(x_0)=f(x_0)$,
\item[(ii)] $|w^+(x)-f(x_0)|\leq C|x-x_0|^{\alpha/2}$ and $|w^-(x)-f(x_0)|\leq C|x-x_0|^{\alpha/2}$ for every $x\in U$. 
\item[(iii)] $|\nabla w^+|>0$ and $|\nabla w^-|>0$ in $U$.
\item[(iv)] $w^-\leq u\leq w^+$ on $\partial U$.
\item[(v)] $\cL w^+<0<\cL w^-$ in $U$.
\end{itemize}
By applying Lemma \ref{barriers} to $w^+$ and $w^-$, we obtain the inequality $w^-\leq u\leq w^+$ in $U$.
This accomplishes the proof by the property (ii).

In order to construct the function $w^+$, notice that $d\in C^2(\{x:0\leq d(x)<\delta_0\})$ for some $\delta_0>0$, because $\partial\Omega\in C^2$. 
We choose $\delta<\delta_0$ and let
\begin{align*}
&v(x)=|x-x_0|^2+\lambda d(x),\\
&w^+(x)=Kv^{\alpha/2}(x)+f(x_0),
\end{align*}
where $K$ and $\lambda$ are to be determined. Obviously (i) and (ii) are valid. To establish (iii), observe that
\begin{align*}
|\nabla w^+|&=K\frac\alpha2v^{\frac\alpha2-1}|\nabla v|,\\
|\nabla v|&=|2(x-x_0)+\lambda\nabla d|\geq\lambda|\nabla d|-2|x-x_0|\\
&\geq\lambda-2|x-x_0|>0,
\end{align*}
provided $\lambda>2\delta$. We also have used the fact that $|\nabla d|=1$ in the last relation. 

For (iv) on $\partial B(x_0,\delta)\cap\Omega$, we have $w^+(x)\geq K\delta^\alpha\geq\|u\|_{C^0(\bar\Omega)}$ if $K$ is chosen large enough. On $\partial\Omega\cap B(x_0,\delta)$ we have 
$$u(x)=f(x)\leq f(x_0)+ \|f\|_{C^{0,\alpha}}|x-x_0|^\alpha\leq w^+(x),$$
provided $K\geq \|f\|_{C^{0,\alpha}(\partial\Omega)}$.

To establish (v), we note that $\phi_\xi(x,tp)=\phi_\xi(x,p)$ and $Dw^+=K\frac\alpha2v^{\frac\alpha2-1}Dv$, then
\begin{align*}
\cL w^+=&\textsl{div}_x(\phi_\xi(x, Dv))=\textsl{div}_x(\phi_\xi(x, 2(x-x_0)+\lambda Dd(x)))\\
=&\textsl{div}_x(\phi_\xi(x, \frac2\lambda(x-x_0)+ Dd(x))).
\end{align*}
Since $d$ is $C^2$ near the boundary $\partial\Omega$ and satisfies the relation \eqref{barrier-condition}, then for a large value of $\lambda$, we will have uniformly $\cL w^+<0$ in the $\delta$-neighborhood of the boundary.

A similar construction provides function $w^-(x)=-Kv^{\alpha/2}(x)+f(x_0)$ for a suitable positive constant $K$.
\end{proof}

\begin{proof}[Proof of Theorem \ref{Holder regularity}]
For $s<t$, consider the supersets $E_s$, $E_t$ of $u$ and assume that $\textsl{dist}(\partial E_s,\partial E_t)=|x-y|$ where $x\in E_t$ and $y\in E_s$. 
 It is sufficient to show that $|u(x)-u(y)|=|t-s|\leq C|x-y|^{\alpha/2}$ whenever $|x-y|<\delta$, where $\delta$ is given by Lemma \ref{boundary-regularity}.
Observe that $O_t\subset E_t\Subset E_s$. By Proposition \ref{minimizer of distance},  we just have two following cases:

(i) If either $x$ or $y$ belongs to $\partial\Omega$, then our result follows from Lemma \ref{boundary-regularity}.

(ii) $x\in\partial E_t\cap\bar O_t$, then $u(x)=\psi(x)$ and $u(y)\geq\psi(y)$, so
$$0<t-s=u(x)-u(y)\leq\psi(x)-\psi(y)\leq [\psi]_{0,\alpha/2}|x-y|^{\alpha/2}.$$
\end{proof}

\begin{proof}[Proof of Theorem \ref{Lipschitz regularity}]
We just need to modify Lemma \ref{boundary-regularity}, to construct functions  $w^+$ and $w^-$ satisfies the conditions (i)-(v), but in (ii) we must replace 
$$|w^{\pm}(x)-f(x_0)|\leq C|x-x_0|^{\frac{1+\alpha}2}.$$
For this, put 
\begin{equation}\label{lipschitz-eq1}
w^+(x):=Kv^{\frac{1+\alpha}2}(x)+\nabla f(x_0)\cdot(x-x_0)+f(x_0),
\end{equation}
and notice that on $\partial\Omega\cap B(x_0,\delta)$, by the $C^{1,\alpha}$ regularity of $f$ there is a positive constant $C_1$ such that the following inequality is established
$$u(x)=f(x)\leq f(x_0)+\nabla f(x_0)\cdot(x-x_0)+C_1|x-x_0|^{\frac{1+\alpha}2}.$$
Therefore, the relation (iv), $u\leq w^+$, will be obtained provided $K\geq C_1$. The rest of the proof is exactly the same.
\end{proof}
\vspace{.3cm}
\noindent


\end{document}